\documentclass[12pt,a4paper]{amsart}
\usepackage[nocompress]{cite}
\usepackage{listings}
\usepackage{docmute}
\usepackage{amssymb, enumerate, calligra}
\usepackage{amsthm}
\usepackage{amsmath}
\usepackage{amsxtra}
\usepackage{latexsym}
\usepackage{mathrsfs}
\usepackage[all,cmtip]{xy}
\usepackage[all]{xy}
\usepackage{enumitem}
\usepackage{xcolor}
\usepackage{graphicx}
\usepackage{comment}
\usepackage{mathabx,epsfig}

\usepackage[dvipdfmx]{hyperref}
\usepackage{cleveref}

\usepackage{mathtools}
\DeclarePairedDelimiter\ceil{\lceil}{\rceil}

\newtheorem{thm}{Theorem}[section]
\crefname{thm}{Theorem}{Theorems}
\newtheorem{lem}[thm]{Lemma}
\crefname{lem}{Lemma}{Lemmas}
\newtheorem{conj}[thm]{Conjecture}
\crefname{conj}{Conjecture}{Conjectures}

\crefname{claim}{Claim}{Claims}
\newtheorem{prop}[thm]{Proposition}
\crefname{prop}{Proposition}{Propositions}

\crefname{cor}{Corollary}{Corollaries}

\crefname{property}{Property}{Properties}
\newtheorem{que}[thm]{Question}
\crefname{que}{Question}{Questions}

\crefname{task}{Task}{Tasks}

\theoremstyle{definition}
\newtheorem{defn}[thm]{Definition}
\crefname{defn}{Definition}{Definitions}
\newtheorem{rmk}[thm]{Remark}
\crefname{rmk}{Remark}{Remarks}
\newtheorem*{ack}{Acknowledgements}
\numberwithin{equation}{section}

\newtheorem{ex}[thm]{Example}
\crefname{ex}{Example}{Examples}

\def\Q{{\mathbb Q}}
\def\R{{\mathbb R}}
\def\Z{{\mathbb Z}}
\def\P{{\mathbb P}}

\def\G{{\mathbb G}}

\def\QQ{\overline{\mathbb Q}}
\def\p{{ \mathfrak{p}}}     
\def\O{{ \mathcal{O}}}

\def\I{{ \mathcal{I}}}

\def\KK{ \overline{K}}

\newcommand{\e}{\epsilon}
\newcommand{\f}{\varphi}

\renewcommand{\l}{\lambda}

\DeclareMathOperator{\id}{id}

\DeclareMathOperator{\Spec}{Spec}
\DeclareMathOperator{\Exc}{Exc}

\DeclareMathOperator{\lct}{lct} 
\DeclareMathOperator{\ct}{ct} 
\DeclareMathOperator{\codim}{codim} 
\DeclareMathOperator{\discrep}{discrep}

  {\end{list}}

\title[]
{ Vojta's conjecture, heights associated with subschemes, and primitive prime divisors in arithmetic dynamics}
\author{Yohsuke Matsuzawa}
\address{Department of Mathematics, Box 1917, Brown University, Providence, Rhode Island 02912, USA}

\email{\href{mailto:matuzawa@math.brown.edu}{matsuzawa@math.brown.edu}}

\begin{document}

\begin{abstract}

Assuming Vojta's conjecture, we give a sufficient condition for the limit
\[
\lim_{n \to \infty} \frac{h_{Y}(f^{n}(x))}{h_{H}(f^{n}(x))}
\]
is equal to zero, where $f \colon X \longrightarrow X$ is a surjective self-morphism on a smooth projective variety $X$, $h_{H}$ is an ample height function on $X$,
and $h_{Y}$ is a  global height function associated with a closed subscheme $Y \subset X$ of codimension at least two.
Based on this, we propose a conjecture on a sufficient condition for the limit being zero.
We point out that our conjecture implies Dynamical Mordell-Lang conjecture for endomorphisms on $\P^{2}_{\QQ}$.
We also discuss applications of Vojta's conjecture with truncated counting function to the problem of 
existence of primitive prime divisors of coordinates of orbits of $f$ 

\end{abstract}

\maketitle

\setcounter{tocdepth}{1}
\tableofcontents

\section{Introduction}

In arithmetic dynamics, it is important to understand the growth of 
height functions along orbits.
Let $K$ be a number field and let $f \colon X \longrightarrow X$ be a surjective self-morphism on a smooth projective geometrically irreducible variety $X$,
all defined over $K$.
In \cite{ma20}, the author proposed the following question.

\begin{que}\label{q:intro-main}
Let $Y \subset X$ be a proper closed subscheme and let 
\begin{itemize}
\item
$\l_{Y,v}$ be a local height function associated with $Y$ at the place $v$ of $K$;
\item
$h_{Y}$ be a global height function associated with $Y$
\end{itemize}
(see \S \ref{sec:htdef} for the definitions of $\l_{Y,v}$ and $h_{Y}$).

Let $h_{H}$ be a height function associated with an ample divisor $H$ on $X$.
Let $x \in X(K)$ be a point.
Under what conditions on $f, Y$, and $x$, do we have
\begin{enumerate}
\item\label{q:local}
\begin{align*}
\lim_{n \to \infty}\frac{\l_{Y,v}(f^{n}(x))}{h_{H}(f^{n}(x))}=0 \quad ?
\end{align*}

\item\label{q:global}
\begin{align*}
\lim_{n \to \infty}\frac{h_{Y}(f^{n}(x))}{h_{H}(f^{n}(x))}=0 \quad ?
\end{align*}
\end{enumerate}

\end{que}

Let us briefly explain what this question means using a concrete example of a self-morphism on $\P^{2}$.
Let us consider a morphism
\begin{align*}
f \colon \P^{2}_{\Q} \longrightarrow \P^{2}_{\Q}, (s : t : u) \mapsto (s^2+t^2+u^2 : tu  : s^2+tu)
\end{align*}
and the orbit of $x = (1:1:1)$ under $f$ (this particular form of $f$ is not important).
Write $f^{n}(x) = (a(n):b(n):c(n))$ where $a(n), b(n), c(n)$ are coprime integers.
The following is the first $7$ iterates.

\vspace{2mm}

\begin{center}

{\tiny
\begin{tabular}{llllc}
$n$ & $a(n)$ & $b(n)$ & $c(n)$ &   $\gcd(a(n),b(n))$ \\
$0$ & $1$ & $1$ & $1$ &  $1$ \\
$1$ & $3$ & $1$ & $2$ &   $1$ \\
$2$ & $14$ & $2$ & $11$ &   $2$ \\
$3$ & $321$ & $22$ & $218$ &   $1$ \\
$4$ & $151049$ & $4796$ & $107837$ &   $1$ \\
$5$ & $34467620586$ & $517186252$ & $23332986653$ &   $2$ \\
$6$ & $1732712616628784933309$ & $12067499915031094556$ & $1200084368775482077952$ &   $1$ \\
\end{tabular}
}

\end{center}

\vspace{2mm}

First, we see that the number of digits of the coordinates grows exponentially.
This corresponds to the fact that $h_{H}(f^{n}(x)) \gg \ll 2^{n}$
which is actually a very special case of so called Kawaguchi-Silverman conjecture.

Second, the sizes of all coordinates look growing in the same speed.
There is a standard way to deduce such simultaneous growth from 
estimates as in \cref{q:intro-main} (\ref{q:local}) for lines $Y=(s=0)$, $(t=0)$, and $(u=0)$,
cf.\ \cite{sil93}, \cite{ma20}.
This is so called Dynamical Lang-Siegel problem.

Third, as in displayed in the table, the greatest common divisor of the first two coordinates $a(n)$ and $b(n)$
are very small.
\cref{q:intro-main} (\ref{q:global}) addresses this kind of problems.
In fact, we have the following:
\begin{align*}
h_{\{(0:0:1)\}}(a:b:c) = \log \frac{\max\{|a|, |b|, |c|\}}{\max\{|a|, |b|\}} + \log \gcd(a,b)
\end{align*}
where $a, b, c$ are coprime integers.
Because of this formula, $h_{Y}$ is sometimes called generalized greatest common divisor when $Y$ has codimension at least two.
Therefore, an affirmative answer to \cref{q:intro-main} (\ref{q:global}) in this case would give a bound of the following form:
for arbitrary small $\e>0$ 
\begin{align*}
\log \gcd(a(n), b(n)) \leq \e 2^{n}
\end{align*}
for $n$ large enough.
This is actually a deep arithmetic phenomenon.
Indeed, this is equivalent to the famous Bugeaud-Corvaja-Zannier \cite{bcz} (cf.\ \cref{ex:a=e=1}).
There are only few cases that such gcds are proven to be small and all the proofs rely on 
Schmidt subspace theorem (cf.\ \cref{ex:cz,ex:a=e=1,prop:Gmcase}).

Finally, let us use this example to introduce another problem, the existence of primitive prime divisors.
The following is the prime factorization of the first coordinate $a(n)$.

\vspace{2mm}

\begin{center}

\begin{tabular}{ll}
$a(n)$ &  prime factorization of $a(n)$ \\
$1$   & $1$ \\
$3$  & $3$ \\
$14$   & $2 \cdot 7$ \\
$321$   & $3 \cdot 107$ \\
$151049$   & $151049$ \\
$34467620586$   & $2 \cdot 3 \cdot 7 \cdot 19 \cdot 2371 \cdot 18217$ \\
$1732712616628784933309$   & $199 \cdot 8707098576024044891$ \\
\end{tabular}

\end{center}

\vspace{2mm}

We can observe that at least one new prime number appears at every step.
Such new prime factors are called primitive prime divisors (of the sequence $a(n)$).
In terms of height function, $p$ is a primitive prime divisor of $a(n)$ if and only if 
$\l_{H, p}(f^{n}(x))>0$ and $\l_{H,p}(f^{m}(x))=0$ for $m <n$.
Here $H$ is the line defined by $(s=0)$.
We discuss this problem in \S \ref{sec:ppd}.
In particular, we prove the existence of primitive prime divisors assuming Vojta's conjecture.
Besides Vojta's conjecture, an estimate of the form as in \cref{q:intro-main} (\ref{q:global}) plays an important role in the proof.

\vspace{10mm}

There are motivations to consider \cref{q:intro-main}  (\ref{q:global}) arising from the study of growth of 
usual Weil height functions.
The function $h_{Y}$, where $\codim Y\geq 2$,
usually shows up as an ``error term'' when we study the growth of Weil height function.
For example, suppose you have a rational map $\pi \colon X \dashrightarrow Z$ to another variety $Z$
and study the sequence $h_{H_{Z}}(\pi(f^{n}(x)))$, 
where $h_{H_{Z}}$ is a Weil height function associated with an ample divisor $H_{Z}$ on $Z$
(cf.\ \cite{bgs19,bfs20}). 
We can rewrite as $h_{H_{Z}}(\pi(f^{n}(x))) = h_{\pi^{*}H_{Z}}(f^{n}(x)) - h_{I_{\pi}}(f^{n}(x))$
where $I_{\pi}$ is the indeterminacy locus of $\pi$, which is a closed subscheme of $X$ of codimension at least two.
It is usually easier to estimate the first term $h_{\pi^{*}H_{Z}}(f^{n}(x))$ and the difficult part is estimating $h_{I_{\pi}}(f^{n}(x))$.

The similar problem happens in the study of the Kawaguchi-Silverman conjecture for self-rational maps ( see \cite{ks3,sil} for the statement of the conjecture ).
The conjecture predicts the growth rate of the sequence $h_{H}(f^{n}(x))$.
If we have good estimates of the function $h_{I_{f}}$ along orbits, 
it would be helpful for the study of Kawaguchi-Silverman conjecture.
(In this context, we should consider \cref{q:intro-main} for rational self-maps.)

\vspace{10mm}

\cref{q:intro-main} (\ref{q:local}) is expected to be true for arbitrary proper closed subscheme $Y$ with appropriate conditions.
As we explain above, a direct consequence of this type of statement is so called Dynamical Lang-Siegel type theorem (cf. \cite[Corollary 1.20]{ma20}).
In \cite{ma20}, we discuss \cref{q:intro-main}  (\ref{q:local}) and gave sufficient conditions when $\dim Y=0$ unconditionally and
assuming Vojta's conjecture when $\dim Y>0$. 

In this paper, we focus on \cref{q:intro-main}  (\ref{q:global}).
We obviously cannot expect (\ref{q:global}) when $Y$ is a divisor.
When $Y$ has codimension at least two, however,  
$h_{Y}$ can be considered as a generalization of greatest common divisor, as we mentioned above,
and it seems reasonable to expect that $h_{Y}$ is fairly small compared with $h_{H}$.
Indeed, assuming Vojta's conjecture , we give a sufficient condition in terms of natural invariants attached to the dynamical system
(\cref{thm:intro:hYvshH}).
We also point out that there are applications of \cref{q:intro-main} (\ref{q:global}) to
Dynamical Mordell-Lang conjecture for $\P^{2}$ and existence of primitive prime divisors (as mentioned above).
Although most of the theorems in this paper are conditional,
we believe they are still interesting as they reveal possible correct geometric conditions that are essential for
these problems.

\subsection{Main theorems}

To state the our main theorems,
let us introduce two quantities ``$ \alpha_{f}(x)$'' and ``$e(Y)$''.

\begin{defn}
Let $K$ be a number field and $X$ be a smooth projective geometrically irreducible variety over $K$.
Let $h_{H}$ be a height function on $X$ associated with an ample divisor $H$ on $X$.
Let $f \colon X \longrightarrow X$ be a surjective morphism.
For any point $x \in X(\KK)$, \emph{the arithmetic degree of $f$ at $x$} is 
\[
\alpha_{f}(x):= \lim_{n\to \infty} \max\{1, h_{H}(f^{n}(x))\}^{1/n}.
\]
This limit always exists and is independent of the choice of $H$ and $h_{H}$
(cf. \cite{ks3, ks1}) .
\end{defn}

\begin{rmk}
Let $d_{1}(f)$ be the first dynamical degree of $f$, i.e. 
the maximum modulus of eigenvalues of $f^{*} \colon N^{1}(X_{\KK}) \longrightarrow N^{1}(X_{\KK})$,
where $N^{1}(X_{\KK})$ is the group of divisors modulo numerical equivalence.
Then it is know that $ \alpha_{f}(x) \leq d_{1}(f)$ for all $x \in X(\KK)$ (\cite{ks1,ma}) and
conjectured that the equality holds if $x$ has Zariski dense $f$-orbit (Kawaguchi-Silverman conjecture).
\end{rmk}

\begin{ex}
When $X= \P^{N}_{K}$, $ \alpha_{f}(x)= d_{1}(f)$ if the $f$-orbit of $x$ is infinite and $ \alpha_{f}(x)=1$ otherwise.
In this case, the first dynamical degree $d_{1}(f)$ is just the degree of the coprime homogeneous polynomials
defining $f$.
\end{ex}

Arithmetic degree measures the asymptotic growth rate of $h_{H}(f^{n}(x))$.
Indeed, we have the following.

\begin{prop}[{\cite[Theorem 1.1]{sano18}} ]\label{prop:growthht}
Suppose $ \alpha_{f}(x) >1$.
Then there are a non-negative integer $l$ and positive real numbers $C_{1}, C_{2}$ such that
\[
C_{1}n^{l} \alpha_{f}(x)^{n} \leq \max\{1, h_{H}(f^{n}(x))\} \leq C_{2} n^{l} \alpha_{f}(x)^{n} 
\]
for all $n\geq 1$. 
\end{prop}

Let $f \colon X \longrightarrow X$ be a surjective morphism on a smooth projective geometrically irreducible variety $X$ over $K$.
Let $Y \subset X$ be a proper closed subscheme.
We introduce an invariant $e(Y)$ which measures how $f$ is ramified along the iterated preimages of $Y$: 
\[
e(Y)=\lim_{n\to \infty}  \sup \{ e_{f^{n}}(x) \mid f^{n}(x) \in Y \}^{1/n} .
\]
See \S \ref{sec:invariant-e} for details of this invariant.

We prove:

\begin{thm}[\cref{thm:vojta-hYvshH}]\label{thm:intro:hYvshH}
Let $Y \subset X$ be a closed subscheme of codimension at least two.
Let $x \in X(K)$ and suppose $e(Y) < \alpha_{f}(x)$.

Assume Vojta's conjecture (\cref{conj:vojta}).
Fix ample height function $h_{H}$ on $X$ and height $h_{Y}$ associated with $Y$.
Then for any $\e>0$, there is a proper closed subset $Z_{\e} \subset X$ such that
\begin{align*}
\left\{ f^{n}(x)\  \middle|\  \frac{h_{Y}(f^{n}(x))}{h_{H}(f^{n}(x))} \geq \e  \right\} \subset Z_{\e}.
\end{align*}
In particular, if $O_{f}(x)$ is generic (i.e. its intersection with every proper closed subset is finite), then we have
\begin{align*}
\lim_{n \to \infty }\frac{h_{Y}(f^{n}(x))}{h_{H}(f^{n}(x))}  = 0.
\end{align*}

\end{thm}

For comparison, let us recall a theorem from \cite{ma20}.

\begin{thm}[cf. {\cite[Theorem 1.13]{ma20}}]
Let $Y \subset X$ be a closed subscheme (which could be a divisor).
Let $x \in X(K)$ and suppose $e(Y) < \alpha_{f}(x)$.

Assume Vojta's conjecture (\cref{conj:vojta}).
Fix ample height function $h_{H}$ on $X$ and local height $\l_{Y}$ associated with $Y$.
Let $v$ be a place of $K$.
Then for any $\e>0$, there is a proper closed subset $Z_{\e} \subset X$ such that
\begin{align*}
\left\{ f^{n}(x)\  \middle|\  \frac{\l_{Y,v}(f^{n}(x))}{h_{H}(f^{n}(x))} \geq \e  \right\} \subset Z_{\e}.
\end{align*}
In particular, if $O_{f}(x)$ is generic, then we have
\begin{align*}
\lim_{n \to \infty }\frac{\l_{Y,v}(f^{n}(x))}{h_{H}(f^{n}(x))}  = 0.
\end{align*}

\end{thm}

\begin{rmk}
In \cite{ma20}, the theorem stated only for generic orbits but the same proof works for not necessarily generic orbits.
Also, we do not use $e(Y)$ in \cite{ma20}, but the conditions in the above theorem and the ones in \cite[Theorem 1.13]{ma20} 
are equivalent.
\end{rmk}

Based on these theorems, we formulate the following conjecture.

\begin{conj}[cf.\ \cref{conj:subvarht-vs-ampleht}]\label{conj:intr:main}
Let $X$ be a smooth projective geometrically irreducible variety over a number field $K$.
Let $f \colon X \longrightarrow X$ be a surjective morphism.
Let $Y \subset X$ be a proper closed subscheme.
Let $H$ be an ample divisor on $X$.
Fix global height functions $h_{Y}$ and $h_{H}$, and local height function $\l_{Y}$.
Let $v$ be a place of $K$.
Let $x \in X( \overline{K})$ be a point such that $e(Y) < \alpha_{f}(x)$ and $O_{f}(x)$ is Zariski dense.
Then we have

\begin{align*}
\lim_{n \to \infty }\frac{\l_{Y,v}(f^{n}(x))}{h_{H}(f^{n}(x))}  = 0.
\end{align*}
If further $Y$ has codimension at least two, we have
\begin{align*}
\lim_{n \to \infty} \frac{h_{Y}(f^{n}(x))}{h_{H}(f^{n}(x))} = 0.
\end{align*}
\end{conj}

See \cref{Q:ht-ratio-higher-dyn.deg} for a possible weakening of the assumption.

Let us remark that when $X$ is an abelian variety,
the first part of \cref{conj:intr:main} is a direct corollary of Faltings' theorem 
(and in this case we do not need the assumption $e(Y) < \alpha_{f}(x)$).

\begin{prop}
Let $X$ be an abelian variety over $K$.
Let $f \colon X \longrightarrow X$ be a surjective morphism and let $Y \subset X$ be a proper closed subscheme.
Let $h_{H}$ be an ample height function on $X$.
Then for every place $v$ of $K$ and every point $x \in X(K)$ with Zariski dense $f$-orbit, we have
\begin{align*}
\lim_{n \to \infty}\frac{\l_{Y,v}(f^{n}(x))}{h_{H}(f^{n}(x))}=0.
\end{align*}
\end{prop}

\begin{proof}
This follows from \cite[Theorem 2]{fal91} and Dynamical Mordell-Lang conjecture for $f$, which is \'etale.
\end{proof}

In \S \ref{sec:ppd}, we apply \cref{thm:intro:hYvshH} to the problem of existence of primitive prime divisors.

\begin{thm}[\cref{thm:ppd2}]
Let $X, f$ be as above.
Let $D$ be an effective ample divisor on $X$.
Take/fix
\begin{itemize} 
\item local height function $\l_{D}$  associated with $D$;
\item a finite set of places $S$ of $K$ containing all archimedean ones.
\end{itemize}
Suppose $D \cap f^{-i}(D)$ has codimension two for every $i \geq 1$.
Let $x \in X(K)$ be a point such that $ \alpha_{f}(x) > e(D)$.
Assume Vojta's conjecture with truncated counting function (\cref{conj:vojtawithcounting}).
Then there is a proper closed subset $Z \subset X$ such that
\begin{align*}
\left\{ f^{n}(x)\  \middle|\  \txt{$f^{n}(x)$ does not have\\ a primitive prime divisor outside $S$ w.r.t $\l_{D}$}  \right\} \subset Z.
\end{align*}
(see \cref{def:ppd} for the definition of primitive prime divisors in this setting.)
\end{thm}

The idea of the use of Vojta's conjecture with truncated counting function has already appeared in \cite{sil13}.
The idea of our proof is based on the proof of \cite[Theorem 2]{sil13}, but we improve it using 
the invariant $e(D)$ and the height estimate \cref{thm:intro:hYvshH}.

This can be also seen as a higher dimensional generalization of \cite{gnt13},
and height associated with subvarieties are involved as a new ingredient.

\vspace{10mm}

The organization of the paper is as follows.
In \S \ref{sec:htdef}, we review the definition of height functions.
In \S \ref{sec:vojta}, we recall Vojta's conjectures and reformulate it using invariants that measure the singularities of pairs $(X, Y)$,
namely canonical threshold and log canonical threshold.
In \S \ref{sec:invariant-e}, we introduce an invariant $e(Y)$ which measures how badly the self-morphism is ramified along the iterated 
preimages of $Y$. We use this invariant to exclude exceptional cases. 
(For example, if $f$ is a surjective self-morphism on $\P^{1}$ of degree $d\geq 2$
and $Y=\{y\}$ consists of a single point that is $f$-exceptional (i.e.\ the backward orbit of $y$ is finite),  then $e(Y)=d$.)
In \S \ref{sec:hY-vs-hH},  we discuss the limit of $h_{Y}/h_{H}$ along orbits and prove \cref{thm:intro:hYvshH}.
In \S \ref{sec:appl-to-dml}, we point out that  \cref{conj:intr:main} implies
Dynamical Mordell-Lang conjecture for $\P^{2}_{\QQ}$ (\cref{prop:DMLP2}).
In \S \ref{sec:ppd}, we apply Vojta's conjecture and \cref{thm:intro:hYvshH} to the problem of existence of primitive prime divisors.

\vspace{10mm}

\noindent
{\bf Convention}
In this paper, we work over a number field unless otherwise stated.
\begin{itemize}
\item An  \emph{algebraic scheme} over a field $k$ is a separated scheme of finite type over $k$;
\item A  \emph{variety} over a field $k$ is an algebraic scheme over $k$ which is irreducible and reduced;
\item A  \emph{nice variety} over a field $k$ is a smooth projective geometrically irreducible scheme over $k$;
\item Let $X$ be a scheme over a field $k$ and $k \subset k'$ be a field extension. 
 \emph{The base change} $X \times_{\Spec k}\Spec k'$ is denoted by $X_{k'}$.
For an ``object" $A$ on $X$, we sometimes use the notation $A_{k'}$ to express the base change of $A$ to $k'$ 
without mentioning to the definition of the base change if the meaning is clear.
\item For a closed subscheme $Y \subset X$, the  \emph{ideal sheaf defining $Y$} is denoted by $\I_{Y}$;
\item For a self-morphism $f \colon X \longrightarrow X$ of an algebraic scheme over $k$ and a 
point $x$ of $X$ (scheme point or $k'$-valued point where $k'$ is a field contains $k$), the  \emph{$f$-orbit of $x$}
is denoted by $O_{f}(x)$, i.e. $O_{f}(x) = \{ f^{n}(x) \mid n=0,1,2, \dots\}$.
\end{itemize}

\begin{ack}
The author would like to thank Joseph Silverman for discussing this subject with him and 
giving him many suggestions and valuable comments.
He would also like to thank Nicole Looper for answering his questions.
The author is supported by JSPS Overseas Research Fellowship.
He would also like to thank the department of mathematics at Brown University for hosting him during his 
fellowship.
\end{ack}

\section{Height functions}\label{sec:htdef}

We fix notation related to height functions.
See \cite{bg, Lan,hs} for the definitions and basic properties of absolute values and local/global height functions associated with divisors,
and see \cite{sil87} for local/global height associated with subschemes.

Let $K$ be a field with proper set of absolute values $M_{K}$. 
Let $M( \overline{K})$ be the set of absolute values on $ \overline{K}$ which extend absolute values of $M_{K}$.
For any intermediate field $K \subset L \subset \KK$ with $[L:K]<\infty$, 
let $M_{K}(L)=M(L)$ be the set of absolute values on $ L$ which extend absolute values of $M_{K}$.
Note that $M(L)$ is also a proper set of absolute values.

Let $X$ be a projective variety over $K$ and $ \overline{Y}\subset X_{\KK}$ be a proper closed subscheme.
We can equip $\overline{Y}$ with a function, which is called  \emph{a local height function associated with $\overline{Y}$}:
\begin{align*}
\lambda_{\overline{Y}} \colon (X_{ \overline{K}}\setminus Y)(\KK) \times M(\KK) \longrightarrow \R; (x,v) \mapsto \lambda_{\overline{Y}, v}(x).
\end{align*}
Note that this is determined up to $M_{K}$-bounded function.

Each $ \lambda_{\overline{Y},v}$ is a function on $(X_{ \overline{K}}\setminus \overline{Y})( \overline{K})$.
If we write  $\overline{Y}=D_{1}\cap \cdots \cap D_{r}$ for some effective Cartier divisors $D_{i}$ on $X_{ \overline{K}}$, then
\[
 \lambda_{\overline{Y},v} = \min_{1\leq i \leq r}\{ \lambda_{D_{i}, v} \} \quad \text{up to $M_{K}$-bounded function}
\]
where $ \lambda_{D_{i}, v}$ are the usual  (logarithmic) local heights associated with $D_{i}$.

When $\overline{Y}$ is also defined over $K$, let us denote by $Y$ the model over $K$.
In this case, we write $ \lambda_{\overline{Y}} = \lambda_{Y}$ and we can choose it so that the indicated map exists,
\[
\xymatrix{
(X \setminus Y)(L) \times M(\KK) \ar@{^{(}->}[r] \ar[d]_{\id \times (\ )|_{L}} & 
(X\setminus Y)(\KK) \times M(\KK) \ar[r]^(.80){ \lambda_{Y}} & \R \\
(X\setminus Y)(L) \times M(L) \ar@/_{10pt}/[rru]_{\exists} &&
}
\]
for any intermediate field $K \subset L \subset \KK$ with $[L:K]<\infty$.
The induced map $(X\setminus Y)(L)\times M(L) \longrightarrow \R$ is also denoted by $ \lambda_{Y}$:
the image of $(x,v) \in (X\setminus Y)(L)\times M(L)$ is denoted by $ \lambda_{Y,v}(x)$.
{\bf We always take $\l_{Y}$ so that these hold.}

A global height function $h_{Y} \colon (X \setminus Y)(\KK) \longrightarrow \R$
associated with $Y$ is defined by this choice of $ \lambda_{Y,v}$:
\begin{align}\label{eq:glhtY}
h_{Y}(x) = \frac{1}{[L:K]}\sum_{v \in M(L)} [L_{v}:K_{v|_{K}}] \lambda_{Y, v}(x)
\end{align}
for $x \in (X \setminus Y)(L)$.
When $x \in Y( \overline{K})$, we set $h_{Y}(x) = \infty$.

In this paper, 
{\bf when $K$ is a number field, $M_{K}$ is the set of absolute values that are normalized as in \cite[p11 (1.6)]{bg}. }
Namely, if $K=\Q$, then $M_{\Q}=\{|\ |_{p} \mid \text{$p=\infty$ or a prime number} \}$ with
\begin{align*}
 &|a|_{\infty} =
 \begin{cases}
 a \quad \text{if $a\geq0$}\\
 -a \quad \text{if $a<0$}
 \end{cases}
 \\
 & |a|_{p} = p^{-n} \quad \txt{if $p$ is a prime and  $a=p^{n}\frac{k}{l}$ where\\ $k,l$ are non zero integers coprime to $p$.}
\end{align*}
For a number field $K$, $M_{K}$ consists of the following absolute values:
\begin{align*}
|a|_{v} = |N_{K_{v}/\Q_{p}}(a)|_{p}^{1/[K:\Q]} 
\end{align*}
where $v$ is a place of $K$ which restricts to $p = \infty$ or a prime number.

For any intermediate field $K \subset L \subset \KK$ with $[L:K]<\infty$ and
$v \in M_{L}$, let $v_{0} \in M(L)$ be the absolute value that is equivalent to $v$
(i.e.\ $v$ is suitably normalized $v_{0}$).
We define 
\[
\l_{Y,v}:= \frac{[L_{v}: K_{v|_{K}}]}{[L:K]} \l_{Y, v_{0}}.
\]
(In other words, $\l_{Y,v}$ is a local height function defined using the normalized absolute value $v$.)
Under this notation, (\ref{eq:glhtY}) becomes
\begin{align*}
h_{Y}(x) = \sum_{v \in M_{L}} \lambda_{Y, v}(x).
\end{align*}

\section{Vojta's conjectures}\label{sec:vojta}

In this section, we recall Vojta's conjectures and reformulate them into forms which are suitable for our purpose.
We use it in the proof of \cref{thm:vojta-hYvshH,thm:ppd1,thm:ppd2}. 
Let $K$ be a number field.

\subsection{Vojta's conjecture}

\begin{defn}
Let $X$ be a smooth variety over $K$ and let $D$ be an effective Cartier divisor on $X$.
Let $x \in D$ be a scheme point.
We say $D$ has normal crossing at $x$ if there is an \'etale morphism $\f \colon U \longrightarrow X$ and a point $u \in U$
such that $\f(u) = x$ and satisfy the following:
\begin{itemize}
\item[$\star$] there is a regular system of parameters $f_{1},\dots ,f_{n}$ of $\O_{U,u}$ such that the ideal of $\f^{*}D$ at $u$ is generated by $f_{1}\cdots f_{r}$
for some $1 \leq r \leq n$.
\end{itemize}
If $D$ has normal crossing at every point $x \in D$, then we say $D$ has normal crossings or is a normal crossing divisor.

We say $D$ is simple normal crossing divisor (SNC for short) if every irreducible component of $D$ is smooth and
$D$ satisfies the above property $\star$ with $U = X$ at every $x$.

\end{defn}

\begin{rmk}
The followings are equivalent.
\begin{itemize}
\item $D$ is a normal crossing divisor on $X$;
\item $D_{ \overline{K}}$ is a normal crossing divisor on $X_{ \overline{K}}$;
\item for every closed point $x \in D_{ \overline{K}}$, there is a regular system of parameters $f_{1},\dots ,f_{n}$ of the completion 
$\widehat{\O}_{X_{ \overline{K}}, x}$ such that the image of the local equation of $D$ at $x$ is $f_{1}\cdots f_{r}$ for some $1 \leq r \leq n$.
\end{itemize}
\end{rmk}

\begin{defn}
Let $K$ be a number field and $S \subset M_{K}$ a finite subset containing all archimedean ones. 
Let $X$ be a geometrically irreducible projective variety over $K$ and let $Y \subset X$ be a proper closed subscheme.
Fix a local height $\l_{Y}$ associated with $Y$, and define
\[
N_{K, S}^{(1)}(Y, x) = N_{ S}^{(1)}(Y, x) := \sum_{v \notin S} \min\left\{ \l_{Y, v}(x), \log \sharp (\O_{K}/\p_{v})^{1/[K:\Q]}  \right\}
\] 
for $x \in (X \setminus Y)(K)$.
This is independent of the choice of $\l_{Y}$ up to bounded function.
\end{defn}

\begin{conj}[Vojta's conjecture with truncated counting function]\label{conj:vojtawithcounting}
Let $K$ be a number field and $S \subset M_{K}$ a finite subset containing all archimedean ones. 
Let $X$ be a nice variety over $K$ and let $D$ be a SNC divisor on  $X$.
Let $H$ be a big divisor on $X$ and $K_{X}$ a canonical divisor on $X$.
Fix $\l_{D}, h_{D},  h_{K_{X}}, h_{H}$.
Then for any $\e>0$ there is a proper Zariski closed subset $Z \subset X$ such that
\[
N_{K,S}^{(1)}(D, x) \geq h_{D}(x) + h_{K_{X}}(x) - \e h_{H}(x)
\]
for all $x \in (X \setminus Z)(K)$.
\end{conj}

This conjecture is apparently stronger than the following widely known form of Vojta's conjecture.

\begin{conj}[Vojta's conjecture]\label{conj:vojta}
Let $K$ be a number field and $S \subset M_{K}$ a finite subset. 
Let $X$ be a nice variety over $K$ and let $D$ be a SNC divisor on  $X$.
Let $H$ be a big divisor on $X$ and $K_{X}$ a canonical divisor on $X$.
Fix $\l_{D}, h_{K_{X}}, h_{H}$.
Then for any $\e>0$ there is a proper Zariski closed subset $Z \subset X$ such that
\[
\sum_{v \in S}\l_{D,v}(x) + h_{K_{X}}(x) \leq  \e h_{H}(x)
\]
for all $x \in (X \setminus Z)(K)$.
\end{conj}

\begin{rmk}
There are slightly generalized version of \cref{conj:vojtawithcounting,conj:vojta},
\cite[Conjecture 2.1, Conjecture 2.3]{vojta98}, which look at points with bounded degree
and take into account the effect of ramifications of extensions of residue fields. 
For this version, these two conjectures are actually equivalent \cite[Theorem 3.1]{vojta98}.
\end{rmk}

We deduce inequalities on height associated with subschemes from these Vojta's conjectures (\cref{prop:vojta-truncated-lct,prop:vojta-ctversion}).
The idea is the following.
For a closed subscheme $Y \subset X$, we take a projective birational morphism $\pi \colon \widetilde{X} \longrightarrow X$
from a smooth projective variety $ \widetilde{X}$
such that $\pi^{-1}(Y)$ is an effective Cartier divisor and its support as well as all exceptional divisors have simple normal crossings.
Take canonical divisors $K_{ \widetilde{ X}}$ and $K_{X}$ on $ \widetilde{ X}$ and $X$ respectively so that $\pi_{*}K_{ \widetilde{ X}}= K_{X}$.
Then define a divisor $ \Delta$ on $ \widetilde{X}$ by the following equation:
\begin{align*}
K_{ \widetilde{ X}} + \Delta = \pi^{*}K_{X} + c\pi^{-1}(Y),
\end{align*}
where $c$ is a parameter which varies over non-negative real numbers.
For \cref{prop:vojta-truncated-lct} below, we bound the global height associated with $ \Delta$ 
using Vojta's conjecture. In order to apply Vojta's conjecture, we need that the coefficient of $ \Delta$ is less than or equal to $1$.
For \cref{prop:vojta-ctversion} below, we bound the height function
\begin{align*}
h_{K_{ \widetilde{ X}}} = h_{\pi^{*}K_{X}} + c h_{\pi^{-1}(Y)} - h_{ \Delta}
\end{align*}
applying Vojta's conjecture to $ \widetilde{X}$. To get a bound of $c h_{\pi^{-1}(Y)} = c h_{Y}\circ \pi$,
we need that $- h_{ \Delta}$ is bounded below (outside a closed set).
This leads us to the condition $- \Delta$ is effective, in other words, the coefficient of $ \Delta$ is less than or equal to $0$.

These conditions on the coefficient of $ \Delta$ exactly correspond to the notion of
\emph{log canonical} and \emph{canonical} singularities.

\subsection{Canonical and log canonical thresholds}

We quickly review the definitions of canonical and log canonical singularities.
We refer to, for example, \cite[Chapter 2]{kolmor} or \cite[Chapter 3]{dem}
for the basic properties of them.
Although these notions are usually defined for pairs of a variety and a divisor on it,
we need the definitions for pairs of a variety and a closed subscheme.

Let $k$ be an algebraically closed field of characteristic zero.
Let $X$ be a smooth quasi-projective variety over $k$ and let $Y \subset X$ be a proper closed subscheme.
(Actually, the following definitions work for $X$ that is $\Q$-Gorenstein, i.e.\ for varieties $X$ whose canonical divisor $K_{X}$ is a $\Q$-Cartier divisor.)

\begin{defn}[Prime divisors over $X$]
A prime divisor over $X$ is a prime divisor $E$ on a normal variety $X'$,
where $X'$ is equipped with a projective birational morphism $\pi \colon X' \longrightarrow X$.
If $\pi$ is not isomorphic at the generic point of $E$, then $E$ is called exceptional over $X$.
We identify two prime divisors over $X$ if they define the same discrete valuation on the function field $k(X)$ of $X$.
\end{defn}

\begin{defn}[Discrepancy of a prime divisor over $X$]
Let $c$ be a non-negative real number.
Let $E$ be a prime divisor over $X$.
The \emph{discrepancy of $E$ with respect to $(X, c\cdot Y)$}, which is denoted by $a(E;X,c\cdot Y)$, is defined as follows:
Take a projective birational morphism $\pi \colon X' \longrightarrow X$ such that 
$X'$ is a normal variety, $E$ appears as a prime divisor on $X'$, and the scheme theoretic inverse image $\pi^{-1}(Y)$
is an effective Cartier divisor on $X'$.
Take canonical divisors $K_{X}$ and $K_{X'}$ so that $\pi_{*}K_{X'}=K_{X}$ and define a divisor $ \Delta$ on $X'$ by the equation
\begin{align*}
K_{X'} + \Delta = \pi^{*}K_{X} + c \pi^{-1}(Y).
\end{align*}
Then we define 
\begin{align*}
a(E;X,c\cdot Y) := \text{(coefficient of $E$ in $- \Delta$)}.
\end{align*}
\end{defn}

\begin{rmk}
See for example \cite[\S 2.3]{kolmor} or \cite[\S 3.1]{dem}
for the independence of this definition on the choice of $X'$ and the canonical divisors.
\end{rmk}

\begin{defn}[Discrepancy of $(X, c\cdot Y)$]
The discrepancy of $(X, c \cdot Y)$ is defined by
\begin{align*}
\discrep(X, c \cdot Y) = \inf \left\{ a(E;X,c\cdot Y) \ \middle|\   \txt{$E$ is an exceptional \\prime divisor over $X$}  \right\}.
\end{align*}
\end{defn}

We say $(X, c\cdot Y)$ is \emph{canonical} if $\discrep(X, c \cdot Y) \geq 0$
and \emph{log canonical} if  $\discrep(X, c \cdot Y) \geq -1$.
These are usually defined for pairs $(X, D)$ of a normal variety and a divisor on it.
We adopt the same definition in our situation.
Note that since our variety $X$ is smooth, we have $\discrep(X, 0 \cdot Y) = 1$.

\begin{defn}[Canonical and log canonical thresholds]
The \emph{canonical threshold} and \emph{log canonical threshold} of $(X, Y)$, denoted by
$\ct(X, Y)$ and $\lct(X,Y)$ respectively, are defied by
\begin{align*}
&\ct(X, Y):= \sup\{c \mid \discrep(X, c \cdot Y)\geq 0\}\\
&\lct(X, Y):= \sup\{ c \mid \discrep(X, c \cdot Y)\geq -1 \}.
\end{align*}
\end{defn}

\begin{rmk}
The supremums in the definition are actually maximum and they can be calculated using
a single log resolution of $(X,Y)$ (cf.\ for example \cite[Corollary 2.31]{kolmor}, \cite[Example 9.3.16]{laz2}).
\end{rmk}

\begin{defn}
Let $X$ be a nice variety over $K$.
Let $Y \subset X$ be a proper closed subscheme.
Define canonical threshold and log canonical threshold by taking base change to $ \overline{K}$:
\begin{align*}
& \ct(X, Y) := \ct(X_{ \overline{K}}, Y_{ \overline{K}})\\
& \lct(X, Y) := \lct(X_{ \overline{K}}, Y_{ \overline{K}}).
\end{align*}
\end{defn}

Let us conclude this subsection with a lemma on a relationship between canonical threshold and log canonical threshold.
We use this relation in the proof of \cref{thm:vojta-hYvshH}.

\begin{lem}\label{lem:ct-vs-lct}
Let $X$ be a smooth projective variety over an algebraically closed field $k$ of characteristic zero.
Let $Y \subset X$ be a proper closed subscheme.
Then we have
\begin{align*}
\ct(X, Y) \geq \frac{\lct(X, Y)}{2}.
\end{align*}
In particular, $\ct(X, Y)>0$.
\end{lem}

\begin{proof}
Let $\pi \colon \widetilde{ X} \longrightarrow X$ be a log resolution of $(X, Y)$, that is
$\pi$ is a birational morphism, $ \widetilde{ X}$ is a smooth projective variety over $k$, 
the scheme theoretic inverse image $\pi^{-1}(Y)$ is an effective Cartier divisor, which we denote by $D$, and
$\pi^{-1}(Y) \cup \Exc(\pi)$ has simple normal crossing support.
By blowing up further if necessary, we may assume that any two non $\pi$-exceptional components
of $D$ do not intersect.
Let us fix a canonical divisor $K_{X}$ of $X$ and let $K_{ \widetilde{ X}}$ be the unique canonical divisor of $ \widetilde{ X}$
such that $\pi_{*}K_{ \widetilde{ X}} = K_{X}$.
Since $X$ is smooth, we have
\begin{align*}
K_{ \widetilde{X}} = \pi^{*}K_{X} + \sum_{i=1}^{r} a_{i}E_{i}
\end{align*}
where $E_{1},\dots ,E_{r}$ are all of the $\pi$-exceptional prime divisors and $a_{i} \in \Z_{>0}$.
We write 
\begin{align*}
D = \sum_{j=1}^{s}d_{j}D_{j} + \sum_{i=1}^{r} b_{i} E_{i}
\end{align*}
where $b_{i} \in \Z_{\geq 0}$, $d_{j} \in \Z_{>0}$, and $D_{j}$ are prime divisors that are not $\pi$-exceptional.
Note that $D_{j}$'s are disjoint. 
For $c\geq 0$, set
\begin{align*}
\Delta:=\pi^{*}K_{X} + cD - K_{ \widetilde{ X}} = c\sum_{j=1}^{s}d_{j}D_{j} + c\sum_{i=1}^{r} b_{i} E_{i} - \sum_{i=1}^{r} a_{i}E_{i}.
\end{align*}
Then $(X, c\cdot Y)$ is canonical (resp. log canonical) if and only if 
\begin{align*}
& cd_{j} \leq 1\quad \text{for all $j=1,\dots, s$, and}\\
& cb_{i}- a_{i} \leq 0 \quad \text{(resp. 1) for all $i =1,\dots ,r$}.
\end{align*}
(cf.\ \cite[Corollary 2.31]{kolmor}.)

Thus we see
\begin{align*}
&\ct(X, Y) = \min_{i,j}\left\{ \frac{1}{d_{j}}, \frac{a_{i}}{b_{i}}\right\};\\
&\lct(X, Y) = \min_{i,j}\left\{ \frac{1}{d_{j}}, \frac{a_{i}+1}{b_{i}}\right\}.
\end{align*}
If $\ct(X, Y) = 1/d_{j}$, then we also have $\lct(X, Y) = 1/d_{j}$ and we are done.
If $\ct(X, Y) = a_{i}/b_{i}$, then we have $\lct(X, Y) \leq (a_{i}+1)/b_{i}$.
Then 
\begin{align*}
2\ct(X, Y) - \lct(X, Y) \geq \frac{a_{i}-1}{b_{i}} \geq 0
\end{align*}
since $a_{i}$ is a positive integer, and we are done.
\end{proof}

\subsection{Vojta's conjecture with canonical and log canonical thresholds}

\begin{prop}\label{prop:vojta-truncated-lct}
Let $K$ be a number field and $S \subset M_{K}$ a finite subset containing all archimedean ones. 
Let $X$ be a nice variety over $K$ and let $Y \subset X$ be a proper closed subscheme.
Let $H$ be a big divisor on $X$ and $K_{X}$ a canonical divisor on $X$.
Fix $\l_{Y}, h_{Y},  h_{K_{X}}, h_{H}$.
Assume Vojta's conjecture with truncated counting function (\cref{conj:vojtawithcounting}).
Then for any $\e>0$ there is a proper Zariski closed subset $Z \subset X$ such that
\[
N_{K,S}^{(1)}(Y, x) \geq \lct(X,Y)h_{Y}(x) + h_{K_{X}}(x) - \e h_{H}(x)
\]
for all $x \in (X \setminus Z)(K)$.
\end{prop}

\begin{lem}\label{lem:basechange-counting}
Let $K$ be a number field and $S \subset M_{K}$ a finite subset containing all archimedean ones. 
Let $X$ be a geometrically irreducible projective variety over $K$ and let $Y \subset X$ be a proper closed subscheme.
Fix a local height function $\l_{Y}$ associated with $Y$.
Let $L \supset K$ be a finite extension and let $T \subset M_{L}$ be the finite set of absolute values that are lying over $S$.
Then 
\begin{align*}
N_{L,T}^{(1)}(Y, x) \leq N_{K,S}^{(1)}(Y, x) 
\end{align*}
for all $x \in (X \setminus Y)(K) \subset (X \setminus Y)(L)$.
\end{lem}
\begin{proof}
We calculate
\begin{align*}
&N_{L,T}^{(1)}(Y, x) \\
&= \sum_{w \notin T} \min\left\{ \l_{Y, w}(x), \log\sharp(\O_{L}/\p_{w})^{1/[L:\Q]} \right\} \\
&= \sum_{v \notin S} \sum_{w|v}
 \min\left\{  \frac{[L_{w}:K_{v}]}{[L:K]}\l_{Y,v}(x) ,  \frac{[ \O_{L}/\p_{w} : \O_{K}/\p_{v}  ]}{[L:K]} \log\sharp(\O_{K}/\p_{v})^{1/[K:\Q]}\right\} \\
& \leq  \sum_{v \notin S} \sum_{w|v} \frac{[L_{w}:K_{v}]}{[L:K]} \min\left\{ \l_{Y, v}(x), \log \sharp (\O_{L}/\p_{v})^{1/[K:\Q]}  \right\} \\
& = \sum_{v \notin S}\min\left\{ \l_{Y, v}(x), \log \sharp (\O_{L}/\p_{v})^{1/[K:\Q]}  \right\}  = N_{K,S}^{(1)}(Y,x).
\end{align*}
\end{proof}

\begin{proof}[Proof of \cref{prop:vojta-truncated-lct}]
Note that the RHS is compatible with base change.
By \cref{lem:basechange-counting}, it is enough to prove the statement after taking a finite base change.
Therefore, we may assume that there is a log resolution $\pi \colon \widetilde{ X} \longrightarrow X$ of $(X,Y)$ over $K$ such that
every irreducible component of $\pi^{-1}(Y)$ and $\Exc(\pi)$ is geometrically irreducible.

Then note that $\pi^{-1}(Y)$ is an effective Cartier divisor.
Set $c = \lct(X,Y)$.
Let $K_{ \widetilde{X}}$ be the canonical divisor such that $\pi_{*}K_{ \widetilde{ X}} = K_{X}$.
If we write
\begin{align*}
K_{ \widetilde{X}} + \Delta = \pi^{*}K_{X} + c \pi^{-1}(Y),
\end{align*}
then $ \Delta$ is a divisor with SNC support and its coefficients are less than or equal to $1$.
Let us write $ \Delta_{\geq 0}$ the effective part of $ \Delta$.
Apply \cref{conj:vojtawithcounting} to $ \widetilde{X}$ and $ \Delta_{\geq 0}$.
Then for any $\e>0$, there is a proper Zariski closed subset $Z' \subset \widetilde{ X}$ such that
\[
N_{K,S}^{(1)}( \ceil{ \Delta_{\geq 0}}, x) \geq h_{ \Delta_{\geq 0}}(x) + h_{K_{ \widetilde{X}}}(x) - \e h_{\pi^{*}H}(x)
\]
for all $x \in ( \widetilde{ X} \setminus Z')(K)$.
Here $\ceil{ \Delta_{\geq 0}}$ is the divisor obtained by rounding up the coefficients of $ \Delta_{\geq 0} $.
Note that outside a proper Zariski closed subset, we have
\begin{align*}
h_{ \Delta_{\geq 0}} + h_{K_{ \widetilde{X}}} - \e h_{\pi^{*}H} &\geq h_{ \Delta} + h_{K_{ \widetilde{X}}} - \e h_{\pi^{*}H} \\
& = h_{K_{X}} \circ \pi + c h_{Y} \circ \pi - \e h_{H}\circ \pi
\end{align*}
up to bounded function.

On the other hand, we have
\begin{align*}
\Delta = \pi^{*}K_{X} + c \pi^{-1}(Y) -K_{ \widetilde{X}} \leq c \pi^{-1}(Y) 
\end{align*}
since $K_{ \widetilde{X}} = \pi^{*}K_{X} + E$ for some effective $\pi$-exceptional divisor $E$.
Since $c\pi^{-1}(Y)$ is effective, we have $c\pi^{-1}(Y) \geq \Delta_{\geq 0}$.
In particular, we have $\pi^{-1}(Y) \geq \ceil{\Delta_{\geq 0}}$.
Thus outside a proper Zariski closed subset, we have
\begin{align*}
N_{K,S}^{(1)}(Y , \pi(-) ) = N_{K,S}^{(1)}(\pi^{-1}(Y) , - ) \geq N_{K,S}^{(1)}( \ceil{\Delta_{\geq 0}}, -)
\end{align*}
up to bounded function.

Therefore  off a proper Zariski closed subset, we get
\begin{align*}
N_{K,S}^{(1)}(Y , - ) \geq h_{K_{X}}  + c h_{Y}  - \e h_{H}
\end{align*}
up to bounded function.
Since the set of $K$-rational points on which $h_{H}$ is bounded is contained in a proper Zariski closed subset,
we can get rid of the implicit constant in the above inequality by applying the same inequality for $\e/2$ for example.

\end{proof}

\begin{prop}\label{prop:vojta-ctversion}
Let $X$ be a nice variety over $K$.
Let $Y \subset X$ be a proper closed subscheme of codimension at least two.
Assume Vojta's conjecture (\cref{conj:vojta}).
Let $H$ be an ample divisor on $X$ and $K_{X}$ a canonical divisor on $X$.
Fix $h_{H}, h_{K_{X}}, h_{Y}$.
Then for any $\e>0$, there is a proper closed subset $Z \subset X$ such that
\begin{align*}
\ct(X, Y)h_{Y}(x) \leq \e h_{H}(x) - h_{K_{X}}(x)
\end{align*}
for $x \in (X \setminus Z)(K)$.
\end{prop}

\begin{proof}
It is enough to prove the statement after taking base change by a finite extension of $K$.
Therefore we may take a log resolution $\pi \colon \widetilde{ X} \longrightarrow X$ of $(X, Y)$, that is
$\pi$ is a birational morphism, $ \widetilde{ X}$ is a nice variety over $K$, 
the scheme theoretic inverse image $\pi^{-1}(Y)$ is an effective Cartier divisor, which we denote by $D$, and
$\pi^{-1}(Y) \cup \Exc(\pi)$ has simple normal crossing support with all irreducible components geometrically integral.
Let $K_{ \widetilde{ X}}$ be the unique canonical divisor of $ \widetilde{ X}$
such that $\pi_{*}K_{ \widetilde{ X}} = K_{X}$.
Let $c = \ct(X, Y)$.
Define 
\[
\Delta := \pi^{*}K_{X} + cD - K_{ \widetilde{ X}}.
\]
Since $Y$ has codimension at least two, $ \Delta$ is $\pi$-exceptional.
Since $c$ is the canonical threshold of $(X, Y)$, all coefficient of $ \Delta$ are less than or equal to zero,
in other words, $- \Delta$ is effective.
Now apply Vojta's conjecture to $ \widetilde{X}$.
Since $\pi^{*}H$ is big, for any $\e$, there is a proper closed subset $Z \subset \widetilde{ X}$ such that
\begin{align*}
h_{K_{ \widetilde{ X}}}(x) \leq \e h_{\pi^{*}H}(x)
\end{align*}
for $x \in ( \widetilde{ X} \setminus Z)(K)$.
Since 
\begin{align*}
h_{ K_{ \widetilde{ X}}} &= c h_{D} + h_{\pi^{*}K_{X}} + h_{- \Delta} + O(1)\\
& =c h_{Y} \circ \pi + h_{K_{X}} \circ \pi  + h_{- \Delta} + O(1)\\
& \geq c h_{Y} \circ \pi + h_{K_{X}} \circ \pi + O(1)
\end{align*}
on $( \widetilde{ X} \setminus (D \cup \Exc(\pi)))$, 
by enlarging $Z$ if necessary, we get
\begin{align*}
ch_{Y}\circ \pi \leq \e h_{H}\circ \pi -  h_{K_{X}} \circ \pi
\end{align*}
on $( \widetilde{ X} \setminus Z)(K)$.
Thus we get
\begin{align*}
ch_{Y} \leq \e h_{H}-  h_{K_{X}} 
\end{align*}
on $(X \setminus \pi(Z))(K)$.

\end{proof}

\section{Invariant $e(Y)$}\label{sec:invariant-e}

Let us introduce an invariant attached to a subscheme which measures how badly self-morphisms ramified along
the iterated preimages of the subscheme.
We briefly summarize the definitions and some of basic properties.
See \cite{ma20} for more properties and proofs of them.

In this section, the ground field is a field of characteristic zero.
Also, if we write $x \in X$ for a scheme $X$, this literally means $x$ is a point of the underlying topological space of $X$.

\begin{defn}
For a finite flat morphism $f \colon X \longrightarrow Y$ between algebraic schemes and a (scheme) point $x \in X$,
we define  \emph{the multiplicity of $f$ at $x$} by
\[
e_{f}(x) = l_{ \mathcal{O}_{X,x}}(\mathcal{O}_{X,x}/f^{*} \mathfrak{m}_{f(x)}\O_{X,x}).
\]
Here $l_{ \mathcal{O}_{X,x}}$ stands for the length as an $ \mathcal{O}_{X,x}$ module.
\end{defn}

\begin{rmk}\label{rmk:multfldext}
Let $k$ be the ground field, which is characteristic zero as we always assume, and $k \subset k'$ be a field extension.
Then for any scheme point $x' \in X_{k'}$ lying over $x \in X$, we have
$e_{f_{k'}}(x')=e_{f}(x)$. 
\end{rmk}

\begin{defn}
Let $X$ be an algebraic scheme over a field of characteristic zero.
Let $f \colon X \longrightarrow X$ be a finite flat surjective morphism.
We define
\begin{align*}
e_{f,-}(x):=e_{-}(x) := \lim_{n\to \infty} \Bigl( \sup\{e_{f^{n}}(y) \mid y\in X, f^{n}(y)=x\}   \Bigr)^{1/n}.
\end{align*}
The existence of this limit is due to \cite[Theorem A.3.5]{gig}. See \cite[Theorem 4.8]{ma20} too.
\end{defn}

\begin{rmk}[{\cite[Theorem A.3.5]{gig}, \cite[Theorem 4.8]{ma20}}]\label{rmk:prope-}
Let us list some properties of the function. 
\begin{enumerate}
\item
The function 
\begin{align*}
e_{f,-} \colon X \longrightarrow \R
\end{align*}
is upper semicontinuous.

\item 
We have
\begin{align*}
e_{f,-}(x) = \max\left\{ e_{f,+}(y) \ \middle| \  \txt{$y\in X$ is an $f$-periodic \\scheme point such that $x \in \overline{\{y\}}$}  \right\}.
\end{align*}
where $e_{f,+}$ is the average multiplicity along the forward orbit:
\begin{align*}
e_{f,+}(x):=e_{+}(x) := \lim_{n\to \infty} e_{f^{n}}(x)^{1/n}.
\end{align*}
\end{enumerate}
\end{rmk}

Finally, we define:
\begin{defn}\label{def:e(Y)}
Let $X$ be an algebraic scheme over a field of characteristic zero.
Let $f \colon X \longrightarrow X$ be a finite flat surjective morphism.
For a closed subscheme $Y \subset X$, we define
\begin{align*}
e(Y) = \max\{ e_{f,-}(y) \mid y \in Y \}.
\end{align*}
The existence of the maximum follows from the upper semicontinuity of the function $e_{f,-}$.
\end{defn}

\begin{rmk}
Let $k$ be the ground field and $k \subset k'$ be a field extension.
Then for any scheme point $x' \in X_{k'}$ lying over $x \in X$, we have
$e_{f_{k'},-}(x')=e_{f,-}(x)$. 
Thus for any closed subscheme $Y \subset X$, we have $e(Y_{k'})=e(Y)$.
\end{rmk}

\begin{rmk}
Although we do not need the following explicitly in this paper, but let us point out the following equality actually holds \cite[Theorem 4.8]{ma20}:
\[
e(Y)=\lim_{n\to \infty} \sup_{y \in Y} \left(\sup \{ e_{f^{n}}(x) \mid f^{n}(x) = y \}^{1/n} \right).
\]

\end{rmk}

\section{Growth of heights associated with subvarieties}\label{sec:hY-vs-hH}

In this section, we discuss the limit
\begin{align*}
\lim_{n \to \infty} \frac{h_{Y}(f^{n}(x))}{h_{H}(f^{n}(x))},
\end{align*}
and prove \cref{thm:intro:hYvshH} (=\cref{thm:vojta-hYvshH}).

Let $K$ be a number field.
Let us first state a conjecture on a sufficient condition for the limit being zero.
This conjecture is justified by \cref{thm:vojta-hYvshH}.

\begin{conj}\label{conj:subvarht-vs-ampleht}
Let $X$ be a nice variety over $K$.
Let $f \colon X \longrightarrow X$ be a surjective morphism.
Let $Y \subset X$ be a closed subscheme of codimension at least two.
Let $H$ be an ample divisor on $X$.
Fix global height functions $h_{Y}$ and $h_{H}$.
Let $x \in X( \overline{K})$ be a point.
If $e(Y) < \alpha_{f}(x)$ and $O_{f}(x)$ is Zariski dense, then we have
\begin{align*}
\lim_{n \to \infty} \frac{h_{Y}(f^{n}(x))}{h_{H}(f^{n}(x))} = 0.
\end{align*}
\end{conj}

\begin{rmk}
As $h_{Y}(f^{n}(x)) = \infty$ if $f^{n}(x) \in Y$, \cref{conj:subvarht-vs-ampleht} contains the statement that 
$O_{f}(x)$ intersects with $Y$ only finitely many times under the assumptions.
This is trivial if $\dim Y=0$.
On the other hand, it is a consequence of Dynamical Mordell-Lang conjecture that
a Zariski dense orbit intersects with every proper closed subset only finitely many times.
\end{rmk}

\begin{rmk}\label{rmk:onthecond}
The assumption that $O_{f}(x)$ is dense is necessary.
However, there seems to be many possible examples satisfying $e(Y) \geq \alpha_{f}(x)$ for which the statement holds
(cf. \cref{ex:a=e=1,prop:hY-vs-hH-etale}).
It might be possible to replace $e(Y) < \alpha_{f}(x)$ with weaker conditions (cf.\ \cref{Q:ht-ratio-higher-dyn.deg}).
\end{rmk}

\begin{ex}\label{ex:cz}
Let $X = \P^{2}_{\Q}$ and 
\[
f \colon \P^{2}_{\Q} \longrightarrow \P^{2}_{\Q}, (X_{0}:X_{1}:X_{2}) \mapsto (X_{0}^{2}:X_{1}^{2}:X_{2}^{2}). 
\]

First consider $Y = \{(0:0:1)\}$ with reduced scheme structure.
Then $e(Y) = 4$.
Let $x = (2:6:1) \in \P^{2}(\Q)$.
Let $h_{\rm naive} = h_{\P^{2}, {\rm naive}}$ be the naive height on $\P^{2}$.
Then we have
\begin{align*}
\frac{h_{Y}(f^{n}(x))}{h_{\rm naive}(f^{n}(x))} = \frac{\log \gcd(2^{2^{n}}, 6^{2^{n}})}{ \log \max\{ 2^{2^{n}}, 6^{2^{n}} \}}
= \frac{2^{n}\log 2}{2^{n} \log 6} = \frac{\log 2}{\log 6} > 0.
\end{align*}
In this case, $x$ has Zariski dense $f$-orbit and $ \alpha_{f}(x) = 2 < 4 = e(Y)$.

Next, consider $Y' = \{(1:1:1)\}$ with reduced scheme structure.
Then $e(Y') = 1$.
Let $x = (a:b:1) \in \P^{2}(\Q)$ where $a, b$ are integers greater than or equal to $2$.
Then we have
\begin{align*}
h_{Y'}(f^{n}(x)) &= h_{Y'}((a^{2^{n}}:b^{2^{n}}:1)) \\
&= \log \frac{ \max\{ |a^{2^{n}}|, |b^{2^{n}}|,1\}}{\max\{ |a^{2^{n}}-1|, |b^{2^{n}} -1|  \}} + \log \gcd(a^{2^{n}}-1, b^{2^{n}}-1),\\
h_{\rm naive}(f^{n}(x)) & = \log \max\{ |a^{2^{n}}|, |b^{2^{n}}| , 1\}.
\end{align*}
Therefore,
\begin{align*}
\lim_{n \to \infty} \frac{h_{Y'}(f^{n}(x))}{h_{\rm naive}(f^{n}(x))} = 0
\end{align*}
is equivalent to 
\begin{align*}
\lim_{n\to \infty}\frac{  \log \gcd(a^{2^{n}}-1, b^{2^{n}}-1)}{\log \max\{|a^{2^{n}}|, |b^{2^{n}}|\}}=0.
\end{align*}
According to Bugeaud-Corvaja-Zannier \cite{bcz}, this holds if $a$ and $b$ are multiplicatively independent, which is equivalent to the Zariski density of $O_{f}(x)$.
\end{ex}

\begin{ex}\label{ex:a=e=1}
Let $a, b$ be multiplicatively independent integers at least $2$.
Let $X = \P^{2}_{\Q}$ and 
\[
f \colon \P^{2}_{\Q} \longrightarrow \P^{2}_{\Q}, (X_{0}:X_{1}:X_{2}) \mapsto (aX_{0}:bX_{1}:X_{2}). 
\]
Let $Y = \{(1:1:1)\}$ with reduced scheme structure and let $x = (1:1:1)$.
Since $f$ is an automorphism on $\P^{2}$, we get $e(Y)=1= \alpha_{f}(x)$.
As $f^{n}(x) = (a^{n}:b^{n}:1)$, the $f$-orbit of $x$ is Zariski dense and
\begin{align*}
&\lim_{n \to \infty} \frac{h_{Y}(f^{n}(x))}{h_{\rm naive}(f^{n}(x))}  \\
&= \lim_{n \to \infty}  \frac{1}{\log \max\{ |a^{n}|,  |b^{n}|, 1\}}  \biggl( \log \frac{ \max\{ |a^{n}|, |b^{n}|,1\}}{\max\{ |a^{n}-1|, |b^{n} -1|  \}} \\
 & \qquad \qquad\qquad\qquad\qquad \qquad\qquad \qquad \qquad  +  \log \gcd(a^{n}-1,b^{n}-1) \biggr) \\
 & = 0
 \end{align*}
 again by \cite{bcz}.

\end{ex}

These examples are generalized to \cref{prop:Gmcase} using a generalization of \cite{bcz}.

Now we prove the main theorem in this section.

\begin{thm}\label{thm:vojta-hYvshH}
Let $X$ be a nice variety over $K$.
Let $f \colon X \longrightarrow X$ be a surjective morphism.
Let $Y \subset X$ be a closed subscheme of codimension at least two.
Let $x \in X(K)$ and suppose $e(Y) < \alpha_{f}(x)$.

Assume Vojta's conjecture \cref{conj:vojta} (for blow ups of $X$).
Fix ample height function $h_{H}$ on $X$ and height $h_{Y}$ associated with $Y$.
Then for any $\e>0$, there is a proper closed subset $Z_{\e} \subset X$ such that
\begin{align*}
\left\{ f^{n}(x)\  \middle|\  \frac{h_{Y}(f^{n}(x))}{h_{H}(f^{n}(x))} \geq \e  \right\} \subset Z_{\e}.
\end{align*}
In particular, if $O_{f}(x)$ is generic (i.e. its intersection with every proper closed subset is finite), then we have
\begin{align*}
\lim_{n \to \infty }\frac{h_{Y}(f^{n}(x))}{h_{H}(f^{n}(x))}  = 0.
\end{align*}

\end{thm}

\begin{rmk}
Dynamical Mordell-Lang conjecture for $f$ says the genericness of $O_{f}(x)$ is equivalent to the Zariski density of $O_{f}(x)$.
\end{rmk}

\begin{rmk}
A very similar result is proven in \cite{huang}.
Huang proved the same conclusion for split self-morphisms on $\P^{1} \times \P^{1}$ assuming Vojta's conjecture.
His assumption on ``$Y$'' is slightly weaker than ``$e(Y) < \alpha_{f}(x)$'' as 
the map has a very concrete form.
\end{rmk}

\begin{proof}[Proof of \cref{thm:vojta-hYvshH}]

Let us fix an arbitrary $\e>0$.

Since $e(Y) < \alpha_{f}(x)$, there is a positive real number $\e'>0$ such that 
$e(Y) + \e' < \alpha_{f}(x)$.
By \cite[Corollary 7.13]{ma20},
we have
\begin{align*}
\lct(X, f^{-k}(Y) \geq \frac{1}{(e(Y) + \e')^{k}}
\end{align*}
for $k \gg 0$. Here $f^{-k}(Y)$ is the scheme theoretic inverse image.
Fix a large $k$ such that the above inequality holds and 
\begin{align*}
\frac{(e(Y)+\e')^{k}}{ \alpha_{f}(x)^{k}} < \e.
\end{align*}

By \cref{prop:vojta-ctversion} (applied to $(X, f^{-k}(Y))$ and ``$\e=1$'') and \cref{lem:ct-vs-lct},
there is a proper closed subset $Z \subset X$ such that
\begin{align*}
h_{f^{-k}(Y)}( y ) \leq 2(e(Y) + \e')^{k}h_{H-K_{X}}(y)
\end{align*}
for all $y \in (X \setminus Z)(K)$.

Thus we get
\begin{align*}
h_{Y}(f^{n}(x)) &\leq h_{f^{-k}(Y)}(f^{n-k}(x)) + C_{1}\\
& \leq 2(e(Y) + \e')^{k}h_{H-K_{X}}(f^{n-k}(x)) + C_{1}
\end{align*}
if $f^{n-k}(x) \notin Z$.
Here the first inequality follows from functoriality of height functions, and $C_{1}$ is a positive constant that is independent of $n$.
Dividing by $h_{H}(f^{n}(x))$, we get
\begin{align*}
\frac{h_{Y}(f^{n}(x))}{h_{H}(f^{n}(x))} \leq \frac{2(e(Y) + \e')^{k}h_{H-K_{X}}(f^{n-k}(x)) + C_{1}}{h_{H}(f^{n}(x))}.
\end{align*}

Since $ \alpha_{f}(x) > 1$ and $H$ is ample, by \cref{prop:growthht}, we have
\begin{align*}
C_{2} n^{a} \alpha_{f}(x)^{n} \leq h_{H}(f^{n}(x)) \leq C_{3} n^{a} \alpha_{f}(x)^{n} ;\\
h_{H-K_{X}}(f^{n}(x)) \leq C_{4} n^{a} \alpha_{f}(x)^{n}
\end{align*}
for some integer $a\geq 0$ and for $n \gg 0$, where $C_{2}, C_{3}, C_{4}$ depend only on $X, f, H,$ and $x$.

Thus for $n$ with $f^{n-k}(x) \notin Z$ and $n \gg 0$, 
\begin{align*}
\frac{h_{Y}(f^{n}(x))}{h_{H}(f^{n}(x))} &\leq \frac{2(e(Y) + \e')^{k} C_{4}(n-k)^{a} \alpha_{f}(x)^{n-k}  }{C_{2} n^{a} \alpha_{f}(x)^{n}} + \frac{C_{1}}{h_{H}(f^{n}(x))}\\
& \leq \frac{2C_{4}}{C_{2}} \e +  \frac{C_{1}}{h_{H}(f^{n}(x))}.
\end{align*}
Since $C_{2}, C_{4}$ are independent of $\e$ and $C_{1}$ is independent of $n$, we are done.

\end{proof}

When the canonical divisor is $\Q$-linearly equivalent to an effective divisor (e.g. Abelian variety, Calabi-Yau variety, etc), 
Vojta's conjecture says the limit goes to zero as soon as the orbit is Zariski dense.

\begin{prop}\label{prop:hY-vs-hH-etale}
Let $X$ be a nice variety over $K$ and suppose $X_{ \overline{K}}$ has non-negative Kodaira dimension.
Let $f \colon X \longrightarrow X$ be a surjective morphism.
Let $Y \subset X$ be a closed subscheme of codimension at least two.

Assume Vojta's conjecture  \cref{conj:vojta} (for blow ups of $X$).
Fix ample height function $h_{H}$ on $X$ and height $h_{Y}$ associated with $Y$.
Let $x \in X(K)$ and suppose $O_{f}(x)$ is Zariski dense.
Then we have
\begin{align*}
\lim_{n \to \infty } \frac{h_{Y}(f^{n}(x))}{h_{H}(f^{n}(x))}  = 0.
\end{align*}
\end{prop}

\begin{proof}
The assumption that $X_{ \overline{K}}$ has non-negative Kodaira dimension is equivalent to say that
$H^{0}(X, nK_{X}) \neq 0$ for some $n>0$, where $K_{X}$ is a canonical divisor of $X$.
In particular, $f$ is \'etale (cf.\ \cite[Lemma 2.3 (2)]{fuj} for example).
Also, height function $h_{K_{X}}$ associated with $K_{X}$ is bounded below outside some proper closed subset, say
\[
h_{K_{X}} \geq C
\]
on $(X \setminus Z)( \overline{K})$ for some $C \in \R$ and $Z \subsetneq X$.

Let $\e>0$ be an arbitrary positive number.
By \cref{prop:vojta-ctversion}, there is a proper closed subset $Z_{\e} \subset X$ such that
\begin{align*}
\ct(X,Y) h_{Y} \leq \e h_{H} - h_{K_{X}}
\end{align*}
on $(X \setminus Z_{\e})(K)$.
Thus we get
\begin{align*}
h_{Y} \leq \frac{1}{\ct(X,Y)} (\e h_{H} - C )
\end{align*}
on $( X \setminus (Z \cup Z_{\e}))(K)$.

Let $x \in X(K)$ be a point with Zariski dense $f$-orbit.
Since Dynamical Mordell-Lang conjecture is true for \'etale morphisms, there is a natural number $n_{0}$ such that 
\begin{align*}
f^{n}(x) \notin Z \cup Z_{\e}
\end{align*}
for $n \geq n_{0}$.
Then for $n \geq n_{0}$,
\begin{align*}
\frac{h_{Y}(f^{n}(x))}{h_{H}(f^{n}(x))}  \leq \frac{1}{ \ct(X,Y)} \left( \e -  \frac{C}{h_{H}(f^{n}(x))} \right).
\end{align*}
Since $h_{H}(f^{n}(x))$ goes to infinity and $\ct(X,Y)$ and $C$ are independent of $n$, we are done.

\end{proof}

\begin{rmk}
In the setting of \cref{prop:hY-vs-hH-etale}, $f$ is \'etale and therefore we have $e(Y)=1$ automatically.
In this case, the proposition says that the ratio $h_{Y}/h_{H}$ goes to zero on the orbit even if $ \alpha_{f}(x) = e(Y)=1$
(under Vojta's conjecture).
\end{rmk}


Let us mention one unconditional results, which is a direct corollary of \cite{lev,yaswang}.

\begin{prop}\label{prop:Gmcase}
Let $Y \subset \G_{m}^{n}$ be a closed subscheme of codimension at least two.
Let $f \colon \G_{m}^{n} \longrightarrow \G_{m}^{n}$ be a surjective morphism.
Let us embed $\G_{m}^{n} \subset \P^{n}, (x_{1}, \dots, x_{n}) \mapsto (x_{1}:\cdots: x_{n}:1)$.
Let $ \overline{Y}$ be the closure of $Y$ in $\P^{n}$ and suppose 
\begin{align*}
(1:0: \cdots :0), (0:1:0:\cdots : 0), \dots, (0:\cdots : 0:1) \notin \overline{Y}.
\end{align*}
Fix height functions $h_{ \overline{Y}}$ associated with $ \overline{Y}$.
Let $h_{ {\rm naive}}$ be the naive height function on $\P^{n}$.
Then for every $x \in \G_{m}^{n}(K)$ with Zariski dense $f$-orbit, we have
\begin{align*}
\lim_{n \to \infty } \frac{h_{ \overline{Y}}(f^{n}(x))}{h_{\rm naive}(f^{n}(x))} = 0.
\end{align*}
\end{prop}
\begin{proof}
For sufficiently large finite set of places $S \subset M_{K}$, we have
\begin{align*}
f^{n}(x) \in \G_{m}^{n}(\O_{K,S}).
\end{align*}
Thus the statement follows from \cite[Theorem 1.3]{yaswang} and Dynamical Mordell-Lang conjecture for $f$, which is \'etale.
\end{proof}

As we mention in \cref{rmk:onthecond}, it might be reasonable to expect that the condition $e(Y)< \alpha_{f}(x)$ can be weaken.
Let us propose one possibility as a question.

\begin{que}\label{Q:ht-ratio-higher-dyn.deg}
Let $X$ be a nice variety over $K$.
Let $f \colon X \longrightarrow X$ be a surjective morphism.
Let $Y \subset X$ be a closed subscheme of codimension at least two.
Let $H$ be an ample divisor on $X$.
Fix global height functions $h_{Y}$ and $h_{H}$.
Let $x \in X( \overline{K})$ be a point.
If $e(Y) < d_{\codim Y}(f)$ and $O_{f}(x)$ is Zariski dense, then we have
\begin{align*}
\lim_{n \to \infty} \frac{h_{Y}(f^{n}(x))}{h_{H}(f^{n}(x))} = 0.
\end{align*}
Here, $d_{\codim Y}(f)$ is the $\codim Y$-th dynamical degree of $f$.
\end{que}

We do not know any counter examples to this question so far, 
but it seems we need completely new techniques to approach this question.

\section{An application of \cref{conj:subvarht-vs-ampleht} to Dynamical Mordell-Lang conjecture}\label{sec:appl-to-dml}

Let us illustrate how \cref{conj:subvarht-vs-ampleht} is strong by deducing Dynamical Mordell-Lang conjecture for $\P^{2}_{\QQ}$.

\begin{prop}\label{prop:DMLP2}
Assume \cref{conj:subvarht-vs-ampleht}.
Let $f \colon \P^{2}_{\QQ} \longrightarrow \P^{2}_{\QQ}$ be a non-isomorphic surjective morphism.
Let $Z \subset \P^{2}_{\QQ}$ be a closed subscheme.
Then for any $x \in \P^{2}(\QQ)$, the set
\[
\{ n \geq 0 \mid f^{n}(x) \in Z\}
\]
is the union of a finite set and finitely many arithmetic progressions.
\end{prop}
\begin{proof}
First of all, if $Z$ has dimension zero or two, then the statement is trivial.
Thus we assume $Z$ is one dimensional.
We may also assume $Z$ is irreducible.
Note that if the orbit $O_{f}(x)$ is not Zariski dense, the statement is trivial again
(the closure of $O_{f}(x)$ intersects with $Z$ at finitely many points or has $Z$ as an irreducible component).
Therefore we may assume $O_{f}(x)$ is Zariski dense, and in particular we have $ \alpha_{f}(x) = d>1$ where $d$ is the first dynamical degree of $f$.
In this case we need to show that $O_{f}(x) \cap Z$ is finite.
Write $Z = (\f_{0} = 0)$ for some non-zero irreducible homogeneous polynomial $\f_{0}$.

Let us consider the set
\begin{align*}
\Sigma = \{z \in Z(\QQ) \mid e_{f,-}(z) \geq \alpha_{f}(x)=d \}.
\end{align*}
By the upper semicontinuity of $e_{f,-}$ (cf.\ \cref{rmk:prope-}),
$ \Sigma$ is (the $\QQ$-points of) a closed subset of $Z$. 
If $ \Sigma=Z(\QQ)$, we have $O_{f}(x) \cap Z=  \emptyset$.
Indeed, if $f^{n}(x) \in Z$ for some $n$, then $e_{f,-}(f^{n}(x)) \geq d >1$ and this implies there is a proper subvariety $P \subset \P^{2}_{\QQ}$
which is $f$-periodic and $f^{n}(x) \in P$ (cf.\ \cref{rmk:prope-}). 
This contradicts with the Zariski density of $O_{f}(x)$.
Therefore, we may assume $ \Sigma$ is a finite set.
Let $\f_{1}$ be a general homogeneous polynomial with $\deg \f_{1} = \deg \f_{0}$ so that
\[
\Sigma \cap (\f_{0} = \f_{1} =0) =  \emptyset.
\]
Consider the rational map
\begin{align*}
\f = (\f_{0} : \f_{1}) \colon \P^{2}_{\QQ} \dashrightarrow \P^{1}_{\QQ}
\end{align*}
defined by these polynomials.
Let $I_{\f} = (\f_{0} = \f_{1} =0)$ be the closed subscheme defined by $\f_{0}$ and $\f_{1}$ (the support is the indeterminacy locus of $\f$).
By the choice of $\f_{1}$, we have $e(I_{\f}) < \alpha_{f}(x)$.
Note also that $f^{n}(x) \notin I_{\f}$ for sufficiently large $n$ since $I_{\f}$ is finite.  
Therefore, for the naive heights $h_{\P^{N}, {\rm naive}}$ on projective spaces and a suitable choice of $h_{I_{\f}}$, we get
\begin{align*}
h_{\P^{1}, {\rm naive}}(\f(f^{n}(x))) &= \deg \f_{0} h_{\P^{2}, {\rm naive}}(f^{n}(x)) - h_{I_{\f}}(f^{n}(x))\\
& = h_{\P^{2}, {\rm naive}}(f^{n}(x)) \left( \deg \f_{0} - \frac{h_{I_{\f}}(f^{n}(x))}{h_{\P^{2}, {\rm naive}}(f^{n}(x))}  \right)
\end{align*}
and this goes to infinity by \cref{conj:subvarht-vs-ampleht}.
This implies there are at most finitely many $n$ such that $f^{n}(x) \notin I_{\f}$ and $\f(f^{n}(x)) = (0:1)$.
Since $Z = (\f_{0}=0)$, we conclude that there are at most finitely many $n$ such that $f^{n}(x) \in Z$.
\end{proof}

\section{Primitive prime divisors in orbits}\label{sec:ppd}

In this section, we discuss existence of primitive prime divisors of coordinates of higher dimensional arithmetic dynamical systems.
Let $K$ be a number field.

\begin{defn}\label{def:ppd}
Let $X$ be a nice variety over $K$.
Let $f \colon X \longrightarrow X$ be a surjective morphism.
Let $D$ be an effective divisor on $X$.
Fix
\begin{itemize} 
\item local height function $\l_{D}$ associated with $D$;
\item a finite set $S \subset M_{K}$ containing all archimedean absolute values.
\end{itemize}
Let $x \in X(K)$ be a point.
A primitive prime divisor of $f^{n}(x)$ outside $S$ with respect to $\l_{D}$ is an absolute value $v \in M_{K} \setminus S$ such that
\begin{align*}
&\l_{D,v}(f^{n}(x)) > 0;\\
&\l_{D,v}(f^{m}(x)) \leq 0 \quad \text{for $0 \leq m < n$}.
\end{align*}
\end{defn}

We refer to \cite{sil13} and \cite[\S 20]{adsurv} for a justification of this definition and the history and survey on
the study of primitive prime divisors in arithmetic dynamics.

\begin{thm}\label{thm:ppd1}
Let $X$ be a nice variety over $K$.
Let $f \colon X \longrightarrow X$ be a surjective morphism.
Let $D$ be an effective ample divisor on $X$.
Take/fix
\begin{itemize} 
\item local height function $\l_{D}$ associated with $D$;
\item a finite set $S \subset M_{K}$ containing all archimedean absolute values.
\end{itemize}
Let $x \in X(K)$ be a point such that $ \alpha_{f}(x) > e(D)$.
Suppose also that $O_{f}(x) \cap D=  \emptyset$.
Assume Vojta's conjecture with truncated counting function (\cref{conj:vojtawithcounting}).
Then there is an integer $l$ and a proper closed subset $Z \subset X$ such that
\begin{align*}
\left\{ f^{n}(x)\  \middle|\  \txt{none of $f^{n}(x), \dots, f^{n+l}(x)$ \\ has a primitive prime divisor outside $S$ w.r.t $\l_{D}$}  \right\} \subset Z.
\end{align*}
\end{thm}

\begin{proof}
Let $e = e(D)$.
Take a small $\e>0$ so that $e+ \e < \alpha_{f}(x)$.
There is a $k_{0}$ such that for $k \geq k_{0}$, we have
\[
\lct(X,(f^{k})^{*}D) \geq \frac{1}{(e+\e)^{k}}.
\]

Fix a global height $h_{D}$.
Let $H$ be an ample divisor on $X$, and let $K_{X}$ be a canonical divisor of $X$.
Fix global height functions $h_{H}$ and $h_{K_{X}}$.
By \cref{prop:vojta-truncated-lct} (applied to $\e=1$), for any $k \geq k_{0}$ there is a proper closed subset $Z_{k} \subset X$ such that
\begin{align*}
& N_{S}^{(1)}(D, f^{n}(x)) \\
& \geq N_{S}^{(1)}((f^{k})^{*}D, f^{n-k}(x)) - C_{1} \\
& \geq \lct(X, (f^{k})^{*}D) h_{(f^{k})^{*}D}(f^{n-k}(x)) + h_{K_{X}}(f^{n-k}(x)) - h_{H}(f^{n-k}(x)) - C_{1}\\
& \geq \frac{1}{(e+ \e)^{k}} h_{D}(f^{n}(x)) - h_{H-K_{X}}(f^{n-k}(x)) - C_{2}
\end{align*}
for all $n\geq k$ if $f^{n-k}(x) \notin Z_{k}$.
Here $C_{i}$ are constants that are independent of $n$ and $h_{H-K_{X}}$ is a height associated with $H-K_{X}$.

Let us consider the sum
\begin{align*}
B_{n,l} = \sum \l_{D,v}(f^{n}(x))
\end{align*}
where the sum runs over $v \in M_{K} \setminus S$ such that
\begin{align*}
&\l_{D,v}(f^{n}(x)) > 0;\\
& \l_{D,v}(f^{m}(x)) \leq 0 \quad \text{for $0 \leq m \leq n-l-1$}.
\end{align*}
Note that if $B_{n,l} > 0$, then one of $f^{n-l}(x), \dots f^{n}(x)$ has a primitive prime divisor outside $S$ with respect to $\l_{D}$.
We calculate
\begin{align*}
&B_{n,l} \\
& \geq N_{S}^{(1)}(D, f^{n}(x)) - \sum_{m = 0}^{n-l-1} \sum_{v \notin S} \l_{D,v}(f^{m}(x))\\
& \geq N_{S}^{(1)}(D, f^{n}(x)) -  \sum_{m = 0}^{n-l-1} h_{D}(f^{m}(x)) - C_{3}\\
& \geq \frac{1}{(e+ \e)^{k}} h_{D}(f^{n}(x)) - h_{H-K_{X}}(f^{n-k}(x))  -  \sum_{m = 0}^{n-l-1} h_{D}(f^{m}(x)) - C_{4}
\end{align*}
for $n\geq k$ if $f^{n-k}(x) \notin Z_{k}$.
Here $C_{i}$ are independent of $n$.

Since $ \alpha_{f}(x) > 1$ and $D$ is ample, by \cref{prop:growthht}, we have
\begin{align*}
C_{5} n^{a} \alpha_{f}(x)^{n} \leq h_{D}(f^{n}(x)) \leq C_{6} n^{a} \alpha_{f}(x)^{n} ;\\
h_{H-K_{X}}(f^{n}(x)) \leq C_{7} n^{a} \alpha_{f}(x)^{n}
\end{align*}
for $n \gg 0$, where $a\geq 0$ is an integer and $C_{5}, C_{6}, C_{7}$ are positive real numbers that are independent of $k$, $n$, and $l$.

Thus we get
\begin{align*}
&B_{n,l}\\
& \geq \frac{h_{D}(f^{n}(x))}{(e+\e)^{k}} \left( 1-  \frac{C_{7} (n-k)^{a} (e+\e)^{k} \alpha_{f}(x)^{n-k}}{C_{5} n^{a} \alpha_{f}(x)^{n}} 
- \sum_{m=0}^{n-l-1} \frac{C_{6} m^{a} \alpha_{f}(x)^{m} }{C_{5} n^{a} \alpha_{f}(x)^{n}} \right) - C_{4}\\
& \geq  \frac{h_{D}(f^{n}(x))}{(e+\e)^{k}} \left( 1- \frac{C_{7}}{C_{5}} \frac{(e+\e)^{k}}{ \alpha_{f}(x)^{k}} 
- \frac{C_{6}}{C_{5}}\frac{1}{ \alpha_{f}(x)^{l}-1} \right) - C_{4}.
\end{align*}
Now, first fix large $k$ so that
\begin{align*}
 \frac{C_{7}}{C_{5}} \frac{(e+\e)^{k}}{ \alpha_{f}(x)^{k}}  \leq \frac{1}{3}.
\end{align*}
Next take and fix a large $l$ so that
\begin{align*}
\frac{C_{6}}{C_{5}}\frac{1}{ \alpha_{f}(x)^{l}-1} \leq \frac{1}{3}.
\end{align*}
Then for $n \gg 0$ such that $f^{n-k}(x) \notin Z_{k}$, we have 
\begin{align*}
B_{n,l} \geq \frac{h_{D}(f^{n}(x))}{3(e+\e)^{k}} - C_{4}>0
\end{align*}
because $C_{4}$ is independent of $n$.
This proves the statement.

\end{proof}

\begin{que}
Are there examples where we actually cannot take $l=0$?
\end{que}

With one additional assumption on $D$, we can show that we can take $l=0$.

\begin{thm}\label{thm:ppd2}
Let $X$ be a nice variety over $K$.
Let $f \colon X \longrightarrow X$ be a surjective morphism.
Let $D$ be an effective ample divisor on $X$.
Take/fix
\begin{itemize} 
\item local height function $\l_{D}$ associated with $D$;
\item a finite set $S \subset M_{K}$ containing all archimedean absolute values.
\end{itemize}
Suppose $D \cap f^{-i}(D)$ has codimension two for every $i \geq 1$.
Let $x \in X(K)$ be a point such that $ \alpha_{f}(x) > e(D)$.
Suppose also that $O_{f}(x) \cap D =  \emptyset$.
Assume Vojta's conjecture with truncated counting function (\cref{conj:vojtawithcounting}).
Then there is a proper closed subset $Z \subset X$ such that
\begin{align*}
\left\{ f^{n}(x)\  \middle|\  \txt{$f^{n}(x)$ does not have\\ a primitive prime divisor outside $S$ w.r.t $\l_{D}$}  \right\} \subset Z.
\end{align*}

\end{thm}

\begin{proof}
First of all, we may assume $O_{f}(x)$ is Zariski dense because otherwise there is nothing to prove.

Let $e = e(D)$.
Take a small $\e>0$ so that $e+ \e < \alpha_{f}(x)$.
There is a $k_{0}$ such that for $k \geq k_{0}$, we have
\[
\lct(X,(f^{k})^{*}D) \geq \frac{1}{(e+\e)^{k}}.
\]

Fix a global height $h_{D}$.
Let $H$ be an ample divisor on $X$, and let $K_{X}$ be a canonical divisor of $X$.
Fix global height functions $h_{H}$ and $h_{K_{X}}$.
By \cref{prop:vojta-truncated-lct} (applied to $\e=1$), for any $k \geq k_{0}$ there is a proper closed subset $Z_{k} \subset X$ such that
\begin{align*}
& N_{S}^{(1)}(D, f^{n}(x)) \\
& \geq N_{S}^{(1)}((f^{k})^{*}D, f^{n-k}(x)) - C_{1} \\
& \geq \lct(X, (f^{k})^{*}D) h_{(f^{k})^{*}D}(f^{n-k}(x)) + h_{K_{X}}(f^{n-k}(x)) - h_{H}(f^{n-k}(x)) - C_{1}\\
& \geq \frac{1}{(e+ \e)^{k}} h_{D}(f^{n}(x)) - h_{H-K_{X}}(f^{n-k}(x)) - C_{2}
\end{align*}
for all $n\geq k$ if $f^{n-k}(x) \notin Z_{k}$.
Here $C_{i}$ are constants that are independent of $n$ and $h_{H-K_{X}}$ is a height associated with $H-K_{X}$.

Let us consider the sum
\begin{align*}
B_{n} = \sum \l_{D,v}(f^{n}(x))
\end{align*}
where the sum runs over $v \in M_{K} \setminus S$ such that
\begin{align*}
&\l_{D,v}(f^{n}(x)) > 0;\\
& \l_{D,v}(f^{m}(x)) \leq 0 \quad \text{for $0 \leq m \leq n-1$}.
\end{align*}
Note that if $B_{n} > 0$, then $f^{n}(x)$ has a primitive prime divisor outside $S$ with respect to $\l_{D}$.

We have
\begin{align*}
&B_{n}\\
& \geq N_{S}^{(1)}(D, f^{n}(x)) - \sum_{\tiny \txt{$v \notin S$, $\l_{D,v}(f^{n}(x))>0$, \\$\l_{D,v}(f^{m}(x) >0$ for some $m<n$}} \log \sharp (\O_{K}/\p_{v})^{1/[K:\Q]} - C_{3}
\end{align*}
where $C_{3}$ is independent of $n$.

For any integer $l\geq 1$, we have
\begin{align*}
&\sum_{\tiny \txt{$v \notin S$, $\l_{D,v}(f^{n}(x))>0$, \\$\l_{D,v}(f^{m}(x) >0$ for some $m<n$}} \log \sharp (\O_{K}/\p_{v})^{1/[K:\Q]}\\
&\leq \sum_{i=1}^{l}h_{D \cap (f^{i})^{*}D}(f^{n-i}(x)) + \sum_{m=0}^{n-l-1} h_{D}(f^{m}(x)) + C_{4}
\end{align*}
where $C_{4}$ is independent of $n$.
Here $h_{D \cap (f^{i})^{*}D}$ is a global height function associated to the codimension two subscheme $D \cap (f^{i})^{*}D$.

Then we get
\begin{align*}
&B_{n}\\
& \geq  \frac{1}{(e+ \e)^{k}} h_{D}(f^{n}(x)) - h_{H-K_{X}}(f^{n-k}(x)) \\
 &-\sum_{i=1}^{l}h_{D \cap (f^{i})^{*}D}(f^{n-i}(x)) - \sum_{m=0}^{n-l-1} h_{D}(f^{m}(x)) - C_{5}
\end{align*}
for $n \geq k, l+1$ and $f^{n-k}(x) \notin Z_{k}$.
Here $C_{5}$ is independent of $n$.

Since $ \alpha_{f}(x) > 1$ and $D$ is ample, by \cref{prop:growthht}, we have
\begin{align*}
C_{6} n^{a} \alpha_{f}(x)^{n} \leq h_{D}(f^{n}(x)) \leq C_{7} n^{a} \alpha_{f}(x)^{n} ;\\
h_{H-K_{X}}(f^{n}(x)) \leq C_{8} n^{a} \alpha_{f}(x)^{n}
\end{align*}
for $n \gg 0$, where $a\geq 0$ is an integer and $C_{6}, C_{7}, C_{8}$ are positive real numbers that are independent of $k$, $n$, and $l$.

Thus we get
\begin{align*}
& B_{n}\\
& \geq \frac{h_{D}(f^{n}(x))}{(e+\e)^{k}} \left( 1- \frac{C_{8} (n-k)^{a} \alpha_{f}(x)^{n-k} (e+\e)^{k}}{C_{6} n^{a} \alpha_{f}(x)^{n}} \right. \\
&\left. - \sum_{i=1}^{l}  \frac{(e+\e)^{k}h_{D \cap (f^{i})^{*}D}(f^{n-i}(x))}{h_{D}(f^{n}(x))}
-\sum_{m=0}^{n-l-1}  \frac{C_{7}m^{a} \alpha_{f}(x)^{m}}{C_{6}n^{a} \alpha_{f}(x)^{n}}  \right) - C_{5}\\
& \geq  \frac{h_{D}(f^{n}(x))}{(e+\e)^{k}} \left( 1 - \frac{C_{8} (e+\e)^{k}}{C_{6} \alpha_{f}(x)^{k}} \right. \\
& \left. -\sum_{i=1}^{l} \frac{C_{7}(e+\e)^{k}}{C_{6}} \frac{h_{D \cap (f^{i})^{*}D}(f^{n-i}(x))}{ h_{D}(f^{n}(x))}  - \frac{C_{7}}{C_{6}( \alpha_{f}(x)^{l}-1)} \right) - C_{5}.
\end{align*}
Now, fix a large $k$ so that
\[
\frac{C_{8} (e+\e)^{k}}{C_{6} \alpha_{f}(x)^{k}} \leq \frac{1}{4}.
\]
Next, fix a large $l$ so that
\[
\frac{C_{7}}{C_{6}( \alpha_{f}(x)^{l}-1)} \leq \frac{1}{4}.
\]
By our assumption, $D \cap (f^{i})^{*}D$ has codimension two and $e(D \cap (f^{i})^{*}D) \leq e(D) < \alpha_{f}(x)$.
By \cref{thm:vojta-hYvshH}, there is a proper closed subset $Z' \subset X$ such that
\[
\sum_{i=1}^{l} \frac{C_{7}(e+\e)^{k}}{C_{6}} \frac{h_{D \cap (f^{i})^{*}D}(f^{n-i}(x))}{ h_{D}(f^{n}(x))} \leq \frac{1}{4}
\]
if $f^{n}(x) \notin Z'$.
Thus, if $n \gg 0$ and $f^{n-k}(x) \notin Z(k)\cup Z'$, we have
\begin{align*}
B_{n} \geq \frac{h_{D}(f^{n}(x))}{4(e+\e)^{k}} - C_{5} > 0
\end{align*}
since $C_{5}$ is independent of $n$.
This finishes the proof.
\end{proof}

\begin{ex}
Let $X = \P^{2}_{\Q}$ and 
\[
f \colon \P^{2}_{\Q} \longrightarrow \P^{2}_{\Q}, (X_{0}:X_{1}:X_{2}) \mapsto (X_{0}^{2}: X_{1}X_{2} : X_{1}^{2}+X_{2}^{2}).
\]
Let $D = (X_{0}=0)$.
Then we can show $e(D) = 2$.
Note that for any point $x=(a:b:c) \in \P^{2}(\Q)$, if it has infinite $f$-orbit, then $ \alpha_{f}(x) = 2 = e(D)$.
As $f^{n}(x)$ is of the form $(a^{2^{n}}: \ast : \ast)$, $f^{n}(x)$ do not have primitive prime divisors 
with respect to $D$ (outside $\{ \infty \}$) for any $x$ and $n\geq 1$.

\end{ex}

\end{document}